\documentclass{amsart}
\usepackage{amssymb, amsmath, amsthm, graphics, comment, xspace, enumerate}
\newtheorem{thm}{Theorem}[section]

\newtheorem{lem}[thm]{Lemma}
\newtheorem{prop}[thm]{Proposition}
\theoremstyle{definition}
\newtheorem{defn}[thm]{Definition}

\theoremstyle{remark}
\newtheorem{rem}[thm]{Remark}
\numberwithin{equation}{section}
\begin{document}
\title[Asymptotically almost periodic and asymptotically...]{Asymptotically almost periodic and asymptotically almost automorphic vector-valued generalized functions}
\author{Marko Kosti\' c}
\address{Faculty of Technical Sciences,
University of Novi Sad,
Trg D. Obradovi\' ca 6, 21125 Novi Sad, Serbia}
\email{marco.s@verat.net}

{\renewcommand{\thefootnote}{} \footnote{
\\ \text{  }  \ \    {\it Key words and phrases.} asymptotically almost periodic distributions,  asymptotically almost periodic ultradistributions, asymptotically almost automorphic distributions, asymptotically almost automorphic ultradistributions,
Banach spaces.
\\  \text{  }  \ \ This research is partially supported by grant 174024 of Ministry
of Science and Technological Development, Republic of Serbia. }}

\begin{abstract}
The main purpose of this paper is to introduce the notion of an asymptotically almost periodic ultradistribution and asymptotically almost automorphic ultradistribution with values in a Banach space, as well as to further analyze the classes of asymptotically almost periodic and asymptotically almost automorphic distributions with values in a Banach space. We
provide some applications of the introduced concepts in the analysis of systems of ordinary differential equations.
\end{abstract}
\maketitle

\section{Introduction and preliminaries}

As it is well-known, the concept of almost periodicity was introduced by Danish mathematician H. Bohr \cite{h.bor} around 1924-1926 and later generalized by many other authors.
Let $I={\mathbb R}$ or $I=[0,\infty),$ and let $f : I \rightarrow X$ be continuous. Given $\epsilon>0,$ we call $\tau>0$ an $\epsilon$-period for $f(\cdot)$ iff\index{$\epsilon$-period}
\begin{align*}
\| f(t+\tau)-f(t) \| \leq \epsilon,\quad t\in I.
\end{align*}
The set constituted of all $\epsilon$-periods for $f(\cdot)$ is denoted by $\vartheta(f,\epsilon).$ It is said that $f(\cdot)$ is almost periodic, a.p. for short, iff for each $\epsilon>0$ the set $\vartheta(f,\epsilon)$ is relatively dense in $I,$ which means that
there exists $l>0$ such that any subinterval of $I$ of length $l$ meets $\vartheta(f,\epsilon)$. The vector space consisting of all almost periodic functions is denoted by $AP(I :X).$\index{function!almost periodic} 

The notion
of a scalar-valued  almost  automorphic  function was introduced by S. Bochner
\cite{bochner} in 1962. In a vector-valued case, definition goes as follows. A bounded continuous function $f : {\mathbb R} \rightarrow X$ is said to be almost automorphic iff for every real sequence $(b_{n})$ there exist a subsequence $(a_{n})$ of $(b_{n})$ and a map $g : {\mathbb R} \rightarrow X$ such that
\begin{align*}
\lim_{n\rightarrow \infty}f\bigl( t+a_{n}\bigr)=g(t)\ \mbox{ and } \  \lim_{n\rightarrow \infty}g\bigl( t-a_{n}\bigr)=f(t),
\end{align*}
pointwise for $t\in {\mathbb R}.$ Due to the famous Bochner's criterion, any almost periodic function has to be almost automorphic; the converse statement is not true, however \cite{diagana}. By $AA({\mathbb R} :X)$ we denote the vector space consisting of all almost automorphic functions.
The
notion of an
almost  automorphic function
on
topological group was introduced and further analyzed in the landmark papers by
W. A. Veech \cite{veech}-\cite{veech-prim} between 1965 and 1967. For more details about almost periodic and almost automorphic functions with values in Banach spaces, we refer the reader to the monographs \cite{diagana} by T. Diagana and \cite{gaston} by G. M. N'Gu\' er\' ekata.  

The notion of a scalar-valued asymptotically almost periodic distribution has been introduced by I. Cioranescu in \cite{ioana-asym}, while the notion of a vector-valued asymptotically almost periodic distribution
has been considered by D. N. Cheban \cite{cheban} following a different approach (see also I. K. Dontvi \cite{dontvi} and A. Halanay, D. Wexler \cite{halanay}). The notion of a scalar-valued almost automorphic distribution and a scalar-valued almost automorphic Colombeau generalized function have been introduced by C. Bouzar, Z. Tchouar \cite{buzar-tress} and C. Bouzar, M. T. Khalladi \cite{buzar-kaladi}. 
Some contributions have been also given by B. Stankovi\' c \cite{stankovic}-\cite{stankovic1}.

In our recent joint research study with D. Velinov \cite{msd-nsjom}, we have analyzed the notions of an almost automorphic distribution and an almost automorphic ultradistribution in Banach space; the notion of an almost periodic ultradistribution in Banach space has been recently analyzed by M. Kosti\' c \cite{mat-bilten} within the framework of Komatsu's theory of ultradistributions, with the corresponding sequences not satisfying the condition (M.3); see also the papers by I. Cioranescu \cite{ioana-tres} and  M. C. G\' omez-Collado \cite{gomezinjo} for first results in this direction. As mentioned in the abstract, the main aim of this paper is to introduce the notions of an asymptotically almost periodic ultradistribution and asymptotically almost automorphic ultradistribution in Banach space, as well as to provide some applications in the qualitative analysis of vector-valued distributional and vector-valued ultradistributional solutions to systems of ordinary differential equations (the notions of an asymptotically almost periodic ultradistribution seem to be not considered elsewhere even in scalar-valued case, while the notion of a vector-valued almost automorphic distribution seems to be completely new, as well). In such a way, we expand and contemplate the results obtained in \cite{buzar-tress}-\cite{buzar-kaladi},
\cite{cheban}, \cite{ioana}-\cite{ioana-asym}, \cite{dontvi}-\cite{halanay}
and \cite{stankovic}-\cite{stankovic1}.

The organization of paper is briefly described as follows. After giving some preliminary results and definitions from the theory of vector-valued ultradistributions (Subsection \ref{vvult}), in Section \ref{collado-automorphic-stankovic} we analyze the notions of asymptotical almost periodicity and asymptotical almost automorphy for vector-valued distributions. Here, we recognize the importance of condition $T\ast \varphi \in AAP({\mathbb R} : X),$ $\varphi \in {\mathcal D}$ ($T\ast \varphi \in AAA({\mathbb R} : X),$ $\varphi \in {\mathcal D}$) for a bounded vector-valued distribution $T\in {\mathcal D}^{\prime}_{L^{1}}(X),$ in contrast with the considerations of I. Cioranescu \cite{ioana-asym} and C. 
Bouzar, F. Z. Tchouar \cite{buzar-tress}, where the above inclusions are required to be valid only for the test functions belonging to the space ${\mathcal D}_{0}.$ Section \ref{collad-stankovic} is devoted to the study of asymptotical almost periodicity and asymptotical almost automorphy for vector-valued ultradistributions. The main result of paper is Theorem \ref{slep-blind-auto-ultra}, where we state an important structural characterization for the class of asymptotically almost periodic (automorphic) vector-valued ultradistributions.
The last section of paper is reserved for certain applications to systems of ordinary differential equations in distribution and ultradistribution spaces.

We use the standard notation throughout the paper. By $(X,\|\cdot \|)$ we denote a non-trivial complex Banach space. The abbreviations $C_{b}(I : X)$ and $C(K:X),$ where $K$ is a non-empty compact subset of ${\mathbb R},$ stand for the spaces consisting of all bounded continuous functions $I \mapsto X$ and all continuous functions $K\mapsto X,$ respectively. Both spaces are Banach endowed with sup-norm. By $C_{0}([0,\infty) :X)$ we denote the closed subspace of $C_{b}([0,\infty) :X)$ consisting of functions vanishing at plus infinity. 

We say that a continuous function $f :{\mathbb R}   \rightarrow X$ is 
asymptotically almost periodic (automorphic) iff there is a function $q\in C_{0}([0,\infty) : X)$ and an almost periodic (automorphic) function $g: {\mathbb R} \rightarrow X$ such that
$f(t)=g(t)+q(t),$ $t\geq 0.$ Let us recall that any (asymptotically) almost periodic function ${\mathbb R} \mapsto X$ is (asymptotically) almost automorphic. By $AAP({\mathbb R} : X),$ resp. $AAA({\mathbb R} : X),$ we denote the vector space consisting of all asymptotically almost periodic, resp. asymptotically almost  automorphic functions.

Let $1\leq p <\infty.$ Then we say that a function $f\in L^{p}_{loc}({\mathbb R} :X)$ is Stepanov $p$-bounded, $S^{p}$-bounded shortly, iff\index{function!Stepanov bounded}
$$
\|f\|_{S^{p}}:=\sup_{t\in {\mathbb R}}\Biggl( \int^{t+1}_{t}\|f(s)\|^{p}\, ds\Biggr)^{1/p}<\infty.
$$
The space $L_{S}^{p}({\mathbb R}:X)$ consisted of all $S^{p}$-bounded functions becomes a Banach space equipped with the above norm.
A function $f\in L_{S}^{p}({\mathbb R}:X)$ is said to be Stepanov $p$-almost periodic, $S^{p}$-almost periodic shortly, iff the function
$
\hat{f} :{\mathbb R} \rightarrow L^{p}([0,1] :X),
$ defined by
$$
\hat{f}(t)(s):=f(t+s),\quad t\in {\mathbb R},\ s\in [0,1]
$$
is almost periodic.
Following H. R. Henr\' iquez \cite{hernan1}, we say that a function $f\in  L_{S}^{p}({\mathbb R}: X)$ is asymptotically Stepanov $p$-almost periodic, asymptotically $S^{p}$-almost periodic shortly, iff there are two locally $p$-integrable functions $g: {\mathbb R} \rightarrow X$ and
$q: [0,\infty)\rightarrow X$ satisfying the following conditions:
\begin{itemize}
\item[(i)] $g(\cdot)$ is $S^p$-almost periodic,
\item[(ii)] $\hat{q}(\cdot)$ belongs to the class $C_{0}([0,\infty) : L^{p}([0,1]:X)),$
\item[(iii)] $f(t)=g(t)+q(t)$ for all $t\geq 0.$
\end{itemize}
By $AAPS^{p}({\mathbb R} : X)$ we denote the space consisting of all  asymptotically Stepanov $p$-almost periodic functions.

The notion of  Stepanov $p$-almost automorphy has been introduced by G. M. N'Gu\'er\'ekata and A. Pankov in \cite{gaston-apankov}: A function $f\in L_{loc}^{p}({\mathbb R}:X)$ is called Stepanov $p$-almost automorphic iff for
every real sequence $(a_{n}),$ there exists a subsequence $(a_{n_{k}})$ \index{ Stepanov almost automorphy}
and a function $g\in L_{loc}^{p}({\mathbb R}:X)$ such that
\begin{align*}
\lim_{k\rightarrow \infty}\int^{t+1}_{t}\Bigl \| f\bigl(a_{n_{k}}+s\bigr) -g(s)\Bigr \|^{p} \, ds =0
\end{align*}
and
\begin{align*}
\lim_{k\rightarrow \infty}\int^{t+1}_{t}\Bigl \| g\bigl( s-a_{n_{k}}\bigr) -f(s)\Bigr \|^{p} \, ds =0
\end{align*}
for each $ t\in {\mathbb R}$; a function $f\in L_{loc}^{p}({\mathbb R}:X)$ is called asymptotically Stepanov $p$-almost automorphic iff there exists an $S^{p}$-almost automorphic
function $g(\cdot)$ and a function $q\in  L_{S}^{p}({\mathbb R}: X)$ such that $f(t)=g(t)+q(t),$ $t\geq 0$ and $\hat{q}\in C_{0}([0,\infty) : L^{p}([0,1]: X)).$ The vector space consisting of all asymptotically $S^{p}$-almost automorphic functions will be denoted by $ AAAS^{p}([0,\infty) : X).$

Concerning distribution spaces, we will use the following elementary notion (cf. L. Schwartz \cite{sch16} for more details). By ${\mathcal D}(X)={\mathcal D}({\mathbb R} : X)$ we denote the Schwartz space of test functions with values in $X$, by ${\mathcal S}(X)={\mathcal S}({\mathbb R} :X)$ we denote the space of rapidly decreasing functions with values in $X$, and by ${\mathcal E}(X)={\mathcal E}({\mathbb R} : X)$ we denote the space of all infinitely differentiable functions with values in $X$; ${\mathcal D}\equiv {\mathcal D}({\mathbb C}),$ 
${\mathcal S}\equiv {\mathcal S}({\mathbb C})$ and ${\mathcal E}\equiv {\mathcal E}({\mathbb C}).$ The spaces of all linear continuous mappings from ${\mathcal D},$ ${\mathcal S}$ and ${\mathcal E}$ into $X$ will be denoted by
${\mathcal D}^{\prime}(X),$ ${\mathcal S}^{\prime}(X)$ and ${\mathcal E}^{\prime}(X),$ respectively. Set ${\mathcal D}_{0}:=\{ \varphi \in {\mathcal D} : \mbox{supp}(\varphi) \subseteq [0,\infty)\}.$

\subsection{Vector-valued ultradistributions}\label{vvult}

In this paper, we will always follow Komatsu's approach to vector-valued ultradistributions, with 
the sequence $(M_p)$ of positive real numbers
satisfying $M_0=1$ and the following conditions:
(M.1): $M_p^2\leq M_{p+1} M_{p-1},\;\;p\in\mathbb{N},$ 
(M.2): $
M_p\leq AH^p\sup_{0\leq i\leq p}M_iM_{p-i},\;\;p\in\mathbb{N},$ for some $A,\ H>1,$ 
(M.3'): $
\sum_{p=1}^{\infty}\frac{M_{p-1}}{M_p}<\infty .
$
Any use of the condition\\ (M.3):
$\sup_{p\in\mathbb{N}}\sum_{q=p+1}^{\infty}\frac{M_{q-1}M_{p+1}}{pM_pM_q}<\infty,$
which is slightly stronger than (M.3$'$), will be explicitly emphasized.

Let us recall that
the Gevrey sequence  $(p!^s)$ satisfies the above conditions ($s>1$). Set
$m_p:=\frac{M_p}{M_{p-1}}$, $p\in\mathbb{N}.$ 

The space of Beurling,
resp., Roumieu ultradifferentiable functions, is
defined by 
$
\mathcal{D}^{(M_p)}
:=\text{indlim}_{K\Subset\Subset\mathbb{R}}\mathcal{D}^{(M_p)}_{K},
$
resp.,
$
\mathcal{D}^{\{M_p\}}
:=\text{indlim}_{K\Subset\Subset\mathbb{R}}\mathcal{D}^{\{M_p\}}_{K},
$
where
$
\mathcal{D}^{(M_p)}_K:=\text{projlim}_{h\to\infty}\mathcal{D}^{M_p,h}_{K},$
resp., 
$\mathcal{D}^{\{M_p\}}_K:=\text{indlim}_{h\to 0}\mathcal{D}^{M_p,h}_{K},
$
\begin{align*}
\mathcal{D}^{M_p,h}_K:=\bigl\{\phi\in C^\infty(\mathbb{R}): \text{supp} \phi\subseteq K,\;\|\phi\|_{M_p,h,K}<\infty\bigr\}
\end{align*}
and
\begin{align*}
\|\phi\|_{M_p,h,K}:=\sup\Biggl\{\frac{h^p|\phi^{(p)}(t)|}{M_p} : t\in K,\;p\in\mathbb{N}_0\Biggr\}.
\end{align*}
Henceforward, the asterisk $*$ is used to denote both, the Beurling case $(M_p)$ or
the Roumieu case $\{M_p\}$. Set ${\mathcal D}^{\ast}_{0}:=\{ \varphi \in {\mathcal D}^{\ast} : \mbox{supp}(\varphi) \subseteq [0,\infty)\}.$
The space
consisted of all continuous linear functions from
$\mathcal{D}^*$ into $X,$ denoted by
$\mathcal{D}^{\prime *}(X):=L(\mathcal{D}^*:X),$ 
is said to be the space of all $X$-valued ultradistributions of $\ast$-class. We also need the notion of space
$\mathcal{E}^*(X),$ defined as  $
\mathcal{E}^{\ast}(X)
:=\text{indlim}_{K\Subset\Subset\mathbb{R}}\mathcal{E}^{\ast}_{K}(X),
$
where in Beurling case
$
\mathcal{E}^{(M_p)}_{K}(X):=\text{projlim}_{h\to\infty}\mathcal{E}^{M_p,h}_{K}(X),$
resp., in Roumieu case
$\mathcal{E}^{\{M_p\}}_{K}(X):=\text{indlim}_{h\to 0}\mathcal{E}^{M_p,h}_{K}(X),
$ and
\begin{align*}
\mathcal{E}^{M_p,h}_{K}(X):=\Biggl\{\phi\in C^\infty(\mathbb{R} :X): \sup_{p\geq 0}\frac{h^{p}\|\phi^{(p)}\|_{C(K : X)}}{M_{p}}<\infty\Biggr\}.
\end{align*}
The space consisting of all linear continuous mappings $\mathcal{E}^*({\mathbb C})\rightarrow X$ is denoted by $\mathcal{E}^{\prime *}(X);$
$\mathcal{E}^{\prime *}:=\mathcal{E}^{\prime *}({\mathbb C}).$
An entire function of the form
$P(\lambda)=\sum_{p=0}^{\infty}a_p\lambda^p$,
$\lambda\in\mathbb{C}$, is of class $(M_p)$, resp., of
class $\{M_p\}$, if there exist $l>0$ and $C>0$, resp., for every
$l>0$ there exists a constant $C>0$, such that $|a_p|\leq Cl^p/M_p$,
$p\in\mathbb{N}$ (\cite{k91}).
The corresponding ultradifferential operator
$P(D)=\sum_{p=0}^{\infty}a_p D^p$ is said to be of class $(M_p)$, resp., of
class $\{M_p\};$ it acts as a continuous linear operator between the spaces ${\mathcal D}^{\ast}$ and ${\mathcal D}^{\ast}$ (${\mathcal D}^{\prime \ast}$ and ${\mathcal D}^{\prime \ast}$).
The convolution of Banach space valued
ultradistributions and scalar-valued ultradifferentiable functions of the same class will be taken in the sense of considerations given on page 685 of \cite{k82}. Let remind ourselves that, for every 
$f\in \mathcal{D}^{\prime *}(X)$ and $\varphi \in \mathcal{D}^*,$ we have $f\ast \varphi \in \mathcal{E}^*(X)$ as well as that the linear mapping $\varphi \mapsto \cdot \ast \varphi : \mathcal{D}^{\prime *}(X) \rightarrow \mathcal{E}^*(X)$ is continuous. The convolution of an $X$-valued
ultradistribution $f(\cdot)$ and a scalar-valued ultradistibution $g\in {\mathcal E}^{\prime \ast }$ with compact support, defined by the identity \cite[(4.9)]{k82}, is an $X$-valued
ultradistribution and the mapping $g \ast \cdot :  \mathcal{D}^{\prime *}(X) \rightarrow \mathcal{D}^{\prime *}(X)$ is continuous. Set $\langle T_{h},\varphi \rangle :=\langle T, \varphi (\cdot -h) \rangle,$ $T\in \mathcal{D}^{\prime *}(X),$ $\varphi \in {\mathcal D}^{\ast}$ ($h>0$). We will use a similar definition for vector-valued distributions.

Assume that the sequence $(M_{p})$ satisfies (M.1), (M.2) and (M.3). Then 
$$
P_{l}(x)=\bigl( 1+x^{2} \bigr)\prod_{p\in {\mathbb N}}\Biggl(1+\frac{x^{2}}{l^{2}m_{p}^{2}}\Biggr),
$$
resp.
$$
P_{r_{p}}(x)=\bigl( 1+x^{2} \bigr)\prod_{p\in {\mathbb N}}\Biggl(1+\frac{x^{2}}{r_{p}^{2}m_{p}^{2}}\Biggr),
$$
defines an ultradifferential operator of class $(M_p)$, resp., of
class $\{M_p\}$; here, $(r_{p})$ is a sequence of positive real numbers tending to infinity. The family consisting of all such sequences will be denoted by ${\mathrm R}$ henceforth. For more details on the subject, the reader may consult \cite{k91}-\cite{k82}.

The spaces of tempered ultradistributions of Beurling,
resp., Roumieu type, are defined by S. Pilipovi\' c \cite{pilip} as duals of the corresponding test spaces
$$
\mathcal{S}^{(M_p)}:=\text{projlim}_{h\to\infty}\mathcal{S}^{M_p,h},\
\mbox{ resp., }\ \mathcal{S}^{\{M_p\}}:=\text{indlim}_{h\to 0}\mathcal{S}^{M_p,h},
$$
where 
\begin{gather*}
\mathcal{S}^{M_p,h}:=\bigl\{\phi\in C^\infty(\mathbb{R}):\|\phi\|_{M_p,h}<\infty\bigr\}\ \ (h>0),
\\
\|\phi\|_{M_p,h}:=\sup\Biggl\{\frac{h^{\alpha+\beta}}{M_\alpha M_\beta}(1+t^2)^{\beta/2}|\phi^{(\alpha)}(t)|:t\in\mathbb{R},
\;\alpha,\;\beta\in\mathbb{N}_0\Biggr\}.
\end{gather*} 
A continuous linear mapping $
\mathcal{S}^{(M_p)} \rightarrow X,$ resp., $
\mathcal{S}^{\{M_p\}}\rightarrow X,$ is said to be an $X$-valued tempered ultradistribution of Beurling, resp., Roumieu type. By $\mathcal{S}^{\prime (M_p)}(X),$ resp. $\mathcal{S}^{\prime \{M_p\} }(X)$ (the common abbreviation will be $\mathcal{S}^{\prime \ast}(X)$), we denote the space consisting of all vector-valued tempered ultradistributions of Beurling, resp., Roumieu type. It is well known that
$\mathcal{S}^{\prime (M_p)}(X) \subseteq \mathcal{D}^{\prime (M_p)}(X)$, resp. $\mathcal{S}^{\prime\{M_p\} }(X) \subseteq \mathcal{D}^{\prime\{M_p\} }(X) .$

Finally, we need some preliminaries concerning the first antiderivative of a vector-valued (ultra-)distribution. 
Let $\eta\in\mathcal{D}_{[1,2]}$  ($\eta\in \mathcal{D}^*_{[1,2]}$) be a fixed test function satisfying $\int_{-\infty}^{\infty}\eta (t)\,dt=1$.
Then, for every fixed $\varphi\in\mathcal{D}$ ($\varphi\in\mathcal{D}^*$), we define $I(\varphi)$ to be
$$
I(\varphi)(x):=\int\limits_{-\infty}^x
\Biggl[\varphi(t)-\eta(t)\int\limits_{-\infty}^{\infty}\varphi(u)\,du\Biggr]\,dt,
\;\;x\in\mathbb{R}.
$$
It can be simply checked that, for every $\varphi\in\mathcal{D}$ ($\varphi\in\mathcal{D}^*$) and $n\in {\mathbb N},$ we have 
$I(\varphi)\in\mathcal{D}$ ($I(\varphi) \in\mathcal{D}^*$), 
$\frac{d^{k}}{dx^{k}}I(\varphi)(x)=\varphi^{(k-1)}(x)-\eta^{(k-1)}(x)\int_{-\infty}^{\infty}\varphi(u)\,du$, $x\in\mathbb{R},$ $k\in {\mathbb N}.$  Define $G^{-1}$ by $G^{-1}(\varphi):=-G(I(\varphi))$, $\varphi\in\mathcal{D}$ ($\varphi\in\mathcal{D}^*$). Then we have
$G^{-1}\in\mathcal{D}'(L(X))$ ($G^{-1}\in {\mathcal D}^{\prime \ast}(L(X))$) and $(G^{-1})'=G$;
more precisely, $-G^{-1}(\varphi')=G(I(\varphi'))=G(\varphi)$, $\varphi\in\mathcal{D}$. For the proof of Theorem \ref{egzist-ultra} below, we will use the following simple fact: if $\langle X, -\varphi^{\prime} \rangle=\langle G, \varphi \rangle,$ $\varphi \in {\mathcal D}_{0}$ ($\varphi \in {\mathcal D}^{\ast}_{0}$) for some $X,\ G\in {\mathcal D}^{\prime}(X)$ ($X,\ G\in {\mathcal D}^{\prime \ast}(X)$), then $X=G^{-1}+\mbox{Const.}$ on $[0,\infty),$ i.e., $\langle X, \varphi \rangle =\langle G^{-1} ,\varphi \rangle +\mbox{Const.} \cdot \int_{0}^{\infty}\varphi(t) \, dt ,$ $\varphi \in {\mathcal D}_{0}$ ($\varphi \in {\mathcal D}^{\ast}_{0}$).

\section[Asymptotical almost periodicity and asymptotical...]{Asymptotical almost periodicity and asymptotical almost automorphy of vector-valued distributions}\label{collado-automorphic-stankovic}

We refer the reader to \cite{buzar-tress}, \cite{ioana} and \cite{msd-nsjom} for the basic results about vector-valued almost periodic distributions and vector-valued almost automorphic distributions.
Let $1\leq p \leq \infty$. Then by ${\mathcal D}_{L^{p}}(X)$ we denote the vector space consisting of all infinitely differentiable functions $f: {\mathbb R} \rightarrow X$ satisfying that for each number $j\in {\mathbb N}_{0}$ we have $f^{(j)}\in L^{p}({\mathbb R} : X).$ The Fr\' echet topology on ${\mathcal D}_{L^{p}}(X)$ is induced by the following system of seminorms
$$
\|f\|_{k}:=\sum_{j=0}^{k}\bigl\|f^{(j)}\bigr\|_{L^{p}({\mathbb R})},\quad f\in {\mathcal D}_{L^{p} (X) } \ \ \bigl( k\in {\mathbb N}\bigr).
$$
If $X={\mathbb C},$ then the above space is simply denoted by ${\mathcal D}_{L^{p}}.$
A linear continuous mapping $f : {\mathcal D}_{L^{1}} \rightarrow X$ is said to be a bounded $X$-valued distribution; the space consisting of such vector-valued distributions will be denoted henceforth by ${\mathcal D}^{\prime}_{L^{1}}(X)$. Endowed with the strong topology, ${\mathcal D}^{\prime}_{L^{1}}(X)$ becomes a complete locally convex space. For every $f\in{\mathcal D}^{\prime}_{L^{1}}(X)$, we have that $f_{| {\mathcal S}} : {\mathcal S} \rightarrow X$ is a tempered $X$-valued distribution (\cite{mat-bilten}).

The space of bounded vector-valued distributions tending to zero at plus infinity, $B^{\prime}_{+,0}(X)$ for short, is defined by
$$
B^{\prime}_{+,0}(X):=\Bigl\{ T\in {\mathcal D}^{\prime}_{L^{1}}(X) : \lim_{h\rightarrow +\infty}\bigl \langle T_{h},\varphi \bigr \rangle=0\mbox{ for all }\varphi \in {\mathcal D}\Bigr\}.
$$
It can be simply verified that the structural characterization for the space $
B^{\prime}_{+,0}({\mathbb C}),$ proved in \cite[Proposition 1]{ioana-asym}, is still valid in vector-valued case as well that the space
$
B^{\prime}_{+,0}({\mathbb C})$ is closed under differentiation.

Let $T\in {\mathcal D}^{\prime}_{L^{1}}(X).$ Then the following assertions are equivalent (\cite{msd-nsjom}):
\begin{itemize}
\item[(i)] $T \ast \varphi \in AP({\mathbb R} : X),$ $\varphi \in {\mathcal D},$ resp., $T \ast \varphi \in AA({\mathbb R} : X),$ $\varphi \in {\mathcal D}.$
\item[(ii)] There exist an integer $k\in {\mathbb N}$ and almost periodic, resp. almost automorphic, functions $f_{j}(\cdot) : [0,\infty) \rightarrow X$ ($1\leq j\leq k$) such that $T=\sum_{j=0}^{k}f_{j}^{(j)}.$
\end{itemize}

We say that a distribution $T\in {\mathcal D}^{\prime}_{L^{1}}(X)$ is almost periodic, resp. almost automorphic, iff $T$ satisfies any of the above two equivalent conditions. By $B^{\prime}_{AP}(X),$ $B^{\prime}_{AA}(X)$ we denote the space consisting of all almost periodic, resp. almost automorphic, distributions.

\begin{defn}\label{definisanie}
A distribution $T\in {\mathcal D}^{\prime}_{L^{1}}(X)$ is said to be asymptotically almost periodic, resp. asymptotically almost automorphic,
iff there exist an almost periodic, resp. almost automorphic, distribution $T_{ap}\in B^{\prime}_{AP}(X),$ resp. $T_{aa}\in B^{\prime}_{AA}(X),$ and a bounded distribution
tending to zero at plus infinity $
Q\in B^{\prime}_{+,0}(X)$ such that $ \langle T ,\varphi \rangle = \langle T_{ap} ,\varphi \rangle + \langle Q ,\varphi \rangle,$ $\varphi \in {\mathcal D}_{0},$
resp. $ \langle T ,\varphi \rangle = \langle T_{aa} ,\varphi \rangle + \langle Q ,\varphi \rangle,$ $\varphi \in {\mathcal D}_{0}.$

By $B^{\prime}_{AAP}(X),$ resp. $B^{\prime}_{AAA}(X),$ we denote the vector space consisting of all asymptotically almost periodic, resp. asymptotically almost automorphic distributions. 
\end{defn}

It is well known that the representation $T=T_{ap}+Q$ is unique in almost periodic case. This is also the case for asymptotical almost automorphy since \cite[Proposition 6, 2.]{buzar-tress} continues to hold in vector-valued case (more precisely, the suppositions  
$T=T_{aa}^{1}+Q=T_{aa}^{2}+Q_{2},$ where $T_{aa}^{1,2}\in B^{\prime}_{AA}(X)$ and $Q_{1,2}\in B^{\prime}_{+,0}(X),$ imply $T_{aa}^{1}=T_{aa}^{2}$ and $Q_{1}=Q_{2}$).

Further on, we would like to observe that the following structural result holds in vector-valued case:

\begin{thm}\label{slep-blind}
Let $T\in {\mathcal D}^{\prime}_{L^{1}}(X).$ Then the following assertions are equivalent:
\begin{itemize}
\item[(i)] $T\in B^{\prime}_{AAP}(X).$
\item[(ii)] $T \ast \varphi \in AAP({\mathbb R} : X),$ $\varphi \in {\mathcal D}_{0}.$ 
\item[(iii)] There exist an integer $k\in {\mathbb N}$ and asymptotically almost periodic functions $f_{j}(\cdot) : {\mathbb R} \rightarrow X$ ($0 \leq j\leq k$) such that $T=\sum_{j=0}^{k}f_{j}^{(j)}$ on $[0,\infty).$
\item[(iv)] There exist an integer $k\in {\mathbb N}$ and bounded asymptotically almost periodic functions $f_{j}(\cdot) : {\mathbb R} \rightarrow X$ ($0\leq j\leq k$) such that $T=\sum_{j=0}^{k}f_{j}^{(j)}$ on $[0,\infty).$
\item[(v)] There exist $S\in {\mathcal D}^{\prime}_{L^{1}}(X),$ $k\in {\mathbb N}$ and bounded asymptotically almost periodic functions $f_{j}(\cdot) : {\mathbb R} \rightarrow X$ ($0\leq j\leq k$) such that $S=\sum_{j=0}^{k}f_{j}^{(j)}$ on ${\mathbb R},$ and $\langle S, \varphi \rangle =\langle T, \varphi \rangle$ for all $\varphi \in {\mathcal D}_{0}.$
\item[(vi)] There exists $S\in {\mathcal D}^{\prime}_{L^{1}}(X)$ such that $\langle S, \varphi \rangle =\langle T, \varphi \rangle$ for all $\varphi \in {\mathcal D}_{0}$ and $S\ast \varphi \in AAP({\mathbb R} : X),$ $\varphi \in {\mathcal D}.$
\item[(vii)] There exists a sequence $(T_{n})$ in ${\mathcal E}(X) \cap AAP({\mathbb R} : X)$ such that $\lim_{n\rightarrow \infty}T_{n}=T$ in ${\mathcal D}_{L^{1}}^{\prime}(X).$ 
\item[(viii)] $T\ast \varphi \in AAP({\mathbb R} : X),$ $\varphi \in {\mathcal D}.$
\end{itemize}
\end{thm}

\begin{proof}
The equivalence of (i)-(iii) can be proved as in scalar-valued case (see \cite[Theorem I, Proposition 1]{ioana-asym}). It is clear that (iv) implies (iii), while the converse statement follows from the fact that there exist functions $g_{j}\in AP({\mathbb R} : X)$ and $q_{j}\in C_{0}([0,\infty) :X)$ for $0\leq j\leq k$ such that
\begin{align*}
\langle T, \varphi \rangle &=\sum_{j=0}^{k}(-1)^{j}\int^{\infty}_{0}f_{j}(t)\varphi^{(j)}(t)\, dt
\\& =\sum_{j=0}^{k}(-1)^{j}\int^{\infty}_{0}\bigl[h_{j}(t)+q_{j,e}(t)\bigr]\varphi^{(j)}(t)\, dt,\quad \varphi \in {\mathcal D}_{0},
\end{align*} 
where $q_{j,e}(\cdot)$ denotes the even extension of function $q_{j}(\cdot)$ to the whole real axis ($0\leq j\leq k$). Therefore, $T=\sum_{j=0}^{k}[h_{j}+q_{j,e}]^{(j)}$ on $[0,\infty)$ and (iii) follows. The implication (iv) $\Rightarrow$ (v) trivially follows from the fact that the expression $S=\sum_{j=0}^{k}f_{j}^{(j)}$
defines an element from ${\mathcal D}^{\prime}_{L^{1}}(X).$ Since the space ${\mathrm A}\equiv AAP({\mathbb R} : X) \cap C_{b}({\mathbb R} : X)$ is uniformly closed (and therefore, $C^{\infty}$-uniformly closed), closed under addition and ${\mathrm A} \ast {\mathcal D} \subseteq {\mathrm A}$ (see \cite{basit-duo-gue} for the notion), an application of \cite[Theorem 2.11]{basit-duo-gue} yields that (v) implies (vi). The implication (vi) $\Rightarrow$ (ii) is trivial, hence we have the equivalence of assertions (i)-(vi). In order to see that (ii) imply (vii), choose
arbitrarily a test function $\rho\in {\mathcal D}$ with supp$(\varphi) \subseteq [0,1].$ Set $\rho_{n}(\cdot):=n\rho(n\cdot),$ for $n\in {\mathbb N}.$ Then (ii) implies $T\ast \rho_{n} \in  
{\mathcal E}(X) \cap AAP({\mathbb R} : X)$ for all $n\in{\mathbb N}.$ Due to the fact that $\lim_{n\rightarrow \infty}T_{n}=T$ in ${\mathcal D}_{L^{1}}^{\prime}(X)$ (see e.g. the second part of proof of \cite[Proposition 7]{buzar-tress}), we get that (vii) holds true. The implication (vii) $\Rightarrow$ (viii) follows from the first part of proof of \cite[Proposition 7]{buzar-tress}, while the implication (viii) $\Rightarrow$ (ii) is trivial,
finishing the proof of theorem.
\end{proof}

Let $1\leq p <\infty,$ and let a function $f\in L^{p}_{loc}({\mathbb R} :X)$ be asymptotically Stepanov $p$-almost periodic. Then the regular distribution associated to $f(\cdot),$ denoted by ${\mathrm f}(\cdot)$ henceforth, is asymptotically almost periodic. In order to see this, let us assume that $p$-integrable functions $g: {\mathbb R} \rightarrow X$ and
$q: [0,\infty)\rightarrow X$ satisfy the conditions from the 
definition of asymptotical Stepanov $p$-almost periodicity. Since 
$$
\int^{\infty}_{-\infty}f(t)\varphi(t)\, dt=\int^{\infty}_{-\infty}[f(t)-g(t)]\varphi(t)\, dt +\int^{\infty}_{-\infty}g(t)\varphi(t)\, dt,\quad \varphi \in {\mathcal D},
$$
and the regular distribution associated to $g(\cdot)$ is almost periodic (\cite{basit-duo-gue}), it suffices to show that the regular distribution associated to $(f-g)(\cdot)$
is 
in class $B^{\prime}_{AAP}(X).$ It can be easily seen that this distribution is bounded, so that Theorem \ref{slep-blind} yields that it is enough to show that, for every fixed $\varphi \in {\mathcal D}_{0},$ we have that $(f-g)\ast \varphi \in AAP({\mathbb R} : X).$ Let supp$(\varphi)\subseteq [a,b]\subseteq [0,\infty),$ where $a,\ b\in {\mathbb N}_{0}.$ Then
\begin{align*}
&[(f-g)\ast \varphi](x)
\\ & =\int^{x-a}_{x-a-1}q(t)\varphi(x-t)\, dt+
\int_{x-a}^{x-a+1}q(t)\varphi(x-t)\, dt +\cdot \cdot \cdot +\int^{x-b+1}_{x-b}q(t)\varphi(x-t)\, dt ,
\end{align*}
for any $\ x\geq b,$ 
and therefore, for any given number $\epsilon>0$ in advance we can find a sufficiently large positive number $x_{0}(\epsilon)\geq b$ such that, for every $x\geq x_{0}(\epsilon),$ we have
$$
\bigl\| [(f-g)\ast \varphi](x) \bigr\| \leq \epsilon,
$$
due to the $S^{p}$-vanishing of function $q(\cdot)$ and the H\"older inequality. 

Therefore, we have proved the following proposition:

\begin{prop}\label{sg-aul}
Let $1\leq p <\infty,$ and let $f\in AAPS^{p}({\mathbb R} : X).$ Then ${\mathrm f}\in B^{\prime}_{AAP}(X).$ 
\end{prop}

The following analogue of Theorem \ref{slep-blind} holds for asymptotical almost automorphy:

\begin{thm}\label{slep-blind-auto}
Let $T\in {\mathcal D}^{\prime}_{L^{1}}(X).$ Then the following assertions are equivalent:
\begin{itemize}
\item[(i)] $T\in B^{\prime}_{AAA}(X).$
\item[(ii)] $T \ast \varphi \in AAA({\mathbb R} : X),$ $\varphi \in {\mathcal D}_{0}.$ 
\item[(iii)] There exist an integer $k\in {\mathbb N}$ and asymptotically almost automorphic functions $f_{j}(\cdot) : {\mathbb R} \rightarrow X$ ($0\leq j\leq k$) such that $T=\sum_{j=0}^{k}f_{j}^{(j)}$ on $[0,\infty).$
\item[(iv)] There exist an integer $k\in {\mathbb N}$ and bounded asymptotically almost automorphic functions $f_{j}(\cdot) : {\mathbb R} \rightarrow X$ ($0\leq j\leq k$) such that $T=\sum_{j=0}^{k}f_{j}^{(j)}$ on $[0,\infty).$
\item[(v)] There exist $S\in {\mathcal D}^{\prime}_{L^{1}}(X),$ $k\in {\mathbb N}$ and bounded asymptotically almost automorphic functions $f_{j}(\cdot) : {\mathbb R} \rightarrow X$ ($0\leq j\leq k$) such that $S=\sum_{j=0}^{k}f_{j}^{(j)}$ on ${\mathbb R},$ and $\langle S, \varphi \rangle =\langle T, \varphi \rangle$ for all $\varphi \in {\mathcal D}_{0}.$
\item[(vi)] There exists $S\in {\mathcal D}^{\prime}_{L^{1}}(X)$ such that $\langle S, \varphi \rangle =\langle T, \varphi \rangle$ for all $\varphi \in {\mathcal D}_{0}$ and $S\ast \varphi \in AAA({\mathbb R} : X),$ $\varphi \in {\mathcal D}.$
\item[(vii)] There exists a sequence $(T_{n})$ in ${\mathcal E}(X) \cap AAA({\mathbb R} : X)$ such that $\lim_{n\rightarrow \infty}T_{n}=T$ in ${\mathcal D}_{L^{1}}^{\prime}(X).$
\item[(viii)] $T\ast \varphi \in AAA({\mathbb R} : X),$ $\varphi \in {\mathcal D}.$
\end{itemize}
\end{thm}

\begin{proof}
The equivalence of (i)-(iii) follows similarly as in the proofs of \cite[Theorem I, Proposition 1]{ioana-asym}, given in almost periodic case.
The equivalence of (iii) and (iv) as well as the fact that (iv) implies (v) can be proved as in the previous theorem.
Since the space ${\mathrm A}\equiv AAA({\mathbb R} : X)  \cap C_{b}({\mathbb R} : X)$ is uniformly closed ($C^{\infty}$-uniformly closed), closed under addition and ${\mathrm A} \ast {\mathcal D} \subseteq {\mathrm A},$ we can apply again \cite[Theorem 2.11]{basit-duo-gue} in order to see that (v) implies (vi). The implication (vi) $\Rightarrow$ (ii) is trivial, so that we have the equivalence of the assertions (i)-(vi). The remaining part of proof can be deduced as in almost periodic case.  
\end{proof}

Let $1\leq p <\infty,$ and let a function $f\in L^{p}_{loc}({\mathbb R} :X)$ be asymptotically Stepanov $p$-almost automorphic. Then the regular distribution associated to $f(\cdot)$ is asymptotically almost automorphic, which can be seen as in the case of asymptotical almost periodicity, with appealing to \cite{buzar-tress} in place of \cite{basit-duo-gue}, for almost automorphic part:

\begin{prop}\label{sg-aul-auto}
Let $1\leq p <\infty,$ and let $f\in AAAS^{p}({\mathbb R} : X).$ Then ${\mathrm f}\in B^{\prime}_{AAA}(X).$ 
\end{prop}

Concerning the assertions of Theorem \ref{slep-blind} and Theorem \ref{slep-blind-auto}, it is worth noting the following:

\begin{rem}\label{boldinjo}
\begin{itemize}
\item[(i)] The validity of (vi) for some $S\in {\mathcal D}^{\prime}_{L^{1}}(X)$ implies its validity for $S$ replaced therein with $S_{Q}=S+Q,$ where $Q \in {\mathcal D}^{\prime}_{L^{1}}(X)$ and supp$(Q) \subseteq (-\infty,0].$ For this, it suffices to observe that $(Q\ast \varphi)(x)=\langle Q, \varphi(x-\cdot)\rangle =0$ for all $x\geq \text{sup}(\text{supp}(\varphi)),$ $\varphi \in {\mathcal D}.$ 
\item[(ii)] Theorem \ref{slep-blind} and Theorem \ref{slep-blind-auto} imply that the spaces $B^{\prime}_{AAP}(X)$ and $B^{\prime}_{AAA}(X)$ are closed under differentiation, as well as that $B^{\prime}_{AAP}(X)$ and $B^{\prime}_{AAA}(X)$ are closed subspaces of ${\mathcal D}^{\prime}_{L^{1}}(X)$ (for this, we can apply the
equivalences of (i) and (ii) in the above-mentioned theorems, the fact that the spaces $AAP({\mathbb R} : X)$ and $AAA({\mathbb R} : X)$ are uniformly closed, and the fact that the translations $\{\varphi (h-\cdot) : h\in {\mathbb R}\}$ of a given function $\varphi \in {\mathcal D}$ forms a bounded subset of ${\mathcal D}_{L^{1}}$). 
\item[(iii)] The implication (i) $\Rightarrow$ (viii) of Theorem \ref{slep-blind} (Theorem \ref{slep-blind-auto}), can be proved directly; see  Remark \ref{basit-kil-osama} below for ultradistibutional case.
\end{itemize}
\end{rem}

\section[Asymptotical almost periodicity and asymptotical...]{Asymptotical almost periodicity and asymptotical almost automorphy of vector-valued ultradistributions}\label{collad-stankovic}

For any $h>0,$ we define
$$
{\mathcal D}_{L^{1}}\bigl((M_{p}),h\bigr):=\Biggl\{ f\in {\mathcal D}_{L^{1}} \ ; \ \|f\|_{1,h}:=\sup_{p\in {\mathbb N}_{0}}\frac{h^{p}\|f^{(p)}\|_{1}}{M_{p}}<\infty \Biggr\} .
$$
Then $({\mathcal D}_{L^{1}}((M_{p}),h),\| \cdot \|_{1,h})$ is a Banach space and the space of all $X$-valued bounded Beurling ultradistributions of class $(M_{p})$, resp., $X$-valued bounded Roumieu ultradistributions of class $\{M_{p}\}$,  
is defined to be the space consisting of all linear continuous mappings from ${\mathcal D}_{L^{1}}((M_{p})),$ resp., $
{\mathcal D}_{L^{1}}(\{M_{p}\}),$ into $X,$ where
$$
{\mathcal D}_{L^{1}}\bigl((M_{p})\bigr):=\text{projlim}_{h\rightarrow +\infty}{\mathcal D}_{L^{1}}\bigl((M_{p}),h\bigr),
$$
resp.,
$$
{\mathcal D}_{L^{1}}\bigl(\{M_{p}\}\bigr):=\text{indlim}_{h\rightarrow 0+}{\mathcal D}_{L^{1}}\bigl((M_{p}),h\bigr).
$$
These spaces, equppied with the strong topologies, will be shortly denoted by $
{\mathcal D}_{L^{1}}^{ \prime }((M_{p}) : X),$ resp., $
{\mathcal D}_{L^{1}}^{ \prime }(\{M_{p}\} : X).$ We will use the shorthand $
{\mathcal D}_{L^{1}}^{ \prime \ast}(M_{p} : X)$ to denote either $
{\mathcal D}_{L^{1}}^{ \prime }((M_{p}) : X)$ or $
{\mathcal D}_{L^{1}}^{ \prime }(\{M_{p}\} : X);$ $
{\mathcal D}_{L^{1}}^{ \prime \ast}(M_{p} )\equiv 
{\mathcal D}_{L^{1}}^{ \prime \ast}(M_{p} : {\mathbb C}).$ As it can be easily proved, $
{\mathcal D}_{L^{1}}^{ \prime \ast}(M_{p} : X)$ is a complete locally convex space. 

It is well known that ${\mathcal D}^{(M_{p})},$ resp. ${\mathcal D}^{\{M_{p}\}},$ is a dense subspace of
${\mathcal D}_{L^{1}}((M_{p}) ),$ resp., $
{\mathcal D}_{L^{1}}(\{M_{p}\} ),$ as well as that  ${\mathcal D}_{L^{1}}((M_{p}) )\subseteq 
{\mathcal D}_{L^{1}}(\{M_{p}\} )$. It can be simply proved that 
$f_{| {\mathcal S}^{(M_{p})}} : {\mathcal S}^{(M_{p})} \rightarrow X,$ resp., $f_{| {\mathcal S}^{\{M_{p}\}}} : {\mathcal S}^{\{M_{p}\}} \rightarrow X,$ is a tempered $X$-valued ultradistribution of class $(M_{p}),$ resp.,
of class $\{M_{p}\}.$ The space $
{\mathcal D}_{L^{1}}^{ \prime \ast }(M_{p} : X)$ is closed under the action of ultradifferential operators of $\ast$-class. 

The space of bounded vector-valued ultradistributions tending to zero at plus infinity, $B^{\prime \ast}_{+,0}(X)$ for short, is defined by
$$
B^{\prime \ast}_{+,0}(X):=\Bigl\{ T\in {\mathcal D}^{\prime \ast }_{L^{1}}(M_{p} : X) : \lim_{h\rightarrow +\infty}\bigl \langle T_{h},\varphi \bigr \rangle=0\mbox{ for all }\varphi \in {\mathcal D}^{\ast}\Bigr\}.
$$

Let $T\in {\mathcal D}^{\prime \ast}_{L^{1}}(X).$ Then we say that $T$ is almost periodic, resp. almost automorphic, iff $T$ satisfies:
$T \ast \varphi \in AP({\mathbb R} : X),$ $\varphi \in {\mathcal D}^{\ast},$ resp., $T \ast \varphi \in AA({\mathbb R} : X),$ $\varphi \in {\mathcal D}^{\ast}.$
By $B^{\prime \ast}_{AP}(X),$ resp. $B^{\prime \ast}_{AA}(X),$ we denote the vector space consisting of all almost periodic, resp. almost automorphic, ultradistributions of $\ast$-class.

\begin{defn}\label{rk-mbo}
An ultradistribution $T\in {\mathcal D}^{\prime \ast}_{L^{1}}(X)$ is said to be asymptotically almost periodic, resp. asymptotically almost automorphic,
iff there exist an almost periodic, resp. almost automorphic, ultradistribution $T_{ap}\in B^{\prime \ast}_{AP}(X),$ resp. $T_{aa}\in B^{\prime \ast}_{AA}(X),$ and a bounded ultradistribution
tending to zero at plus infinity $
Q\in B^{\prime \ast}_{+,0}(X)$ such that $ \langle T ,\varphi \rangle = \langle T_{ap} ,\varphi \rangle + \langle Q ,\varphi \rangle,$ $\varphi \in {\mathcal D}_{0}^{\ast},$
resp. $ \langle T ,\varphi \rangle = \langle T_{aa} ,\varphi \rangle + \langle Q ,\varphi \rangle,$ $\varphi \in {\mathcal D}_{0}^{\ast}.$

By $B^{\prime \ast}_{AAP}(X),$ resp. $B^{\prime \ast}_{AAA}(X),$ we denote the vector space consisting of all asymptotically almost periodic, resp. automorphic, ultradistributions of $\ast$-class.
\end{defn}

As in distribution case, decomposition of an asymptotically almost periodic (automorphic) ultradistribution of $\ast$-class into its almost periodic (automorphic) part and bounded,
tending to zero at plus infinity part, is unique. The space $B^{\prime \ast}_{AAP}(X),$ resp. $B^{\prime \ast}_{AAA}(X),$
is closed under the action of ultradifferential operators of $\ast$-class. This follows from the fact that this is true for the space $B^{\prime \ast}_{AP}(X),$ resp. $B^{\prime \ast}_{AA}(X)$ (see \cite{mat-bilten} and \cite{msd-nsjom}), as well as that, for every 
$ Q\in
B^{\prime \ast}_{+,0}(X)$ and for every ultradifferential operator $P(D)$ of $\ast$-class, we have $\langle P(D) Q,\varphi (\cdot-h)\rangle =\langle Q, [P(D)\varphi](\cdot-h)\rangle,$ $h\in {\mathbb R}$.

For the sequel, we need the following preparation. Let ${\mathrm A}\subseteq {\mathcal D}^{\prime \ast}(X).$ Following B. Basit and H. G\"uenzler \cite{basit-duo-gue}, whose examinations have been carried out in distributional case, we have recently introduced the following notion in \cite{mat-bilten}:
$$
\mathcal{D}^{\prime *}_{{\mathrm A}}(X):=\Bigl\{  T\in\ \mathcal{D}^{\prime *}(X) : T\ast \varphi \in {\mathrm A}\mbox{ for all }\varphi \in \mathcal{D}^{*}\Bigr\}.
$$
Then $
\mathcal{D}^{\prime *}_{{\mathrm A}}(X)=
\mathcal{D}^{\prime *}_{{\mathbb A}\cap B}(X),$ for any set $B\subseteq L_{loc}^{1}({\mathbb R} : X)$ that contains $C^{\infty}({\mathbb R} : X),$ as well as the set $
\mathcal{D}^{\prime *}_{{\mathrm A}}(X)$ is closed under the action of ultradifferential operators of $\ast$-class. Furthermore, the following holds \cite{mat-bilten}:

\begin{itemize}
\item[(i)]
Assume that there exist an ultradifferential operator $P(D)=\sum_{p=0}^{\infty}a_p D^p$
of class $(M_p)$, resp., of
class $\{M_p\},$ and $f,\ g\in \mathcal{D}^{\prime *}_{{\mathrm A}}(X)$ such that $T=P(D)f+g.$ If ${\mathrm A}$ is closed under addition, then
$T\in \mathcal{D}^{\prime *}_{{\mathrm A}}(X).$
\item[(ii)] If ${\mathrm A} \cap C({\mathbb R} : X)$ is closed under uniform convergence, $T\in 
{\mathcal D}_{L^{1}}^{ \prime }((M_{p}) : X)$ and $T\ast \varphi \in {\mathrm A},$ $\varphi \in {\mathcal D}^{(M_{p})},$ then there is a number $h>0$ such that for each compact set $K\subseteq {\mathbb R}$ we have $T\ast \varphi \in {\mathrm A},$ $\varphi \in \mathcal{D}^{M_p,h}_K.$ 
\item[(iii)] Assume that $T\in \mathcal{D}^{\prime (M_{p})}(X)$ and there exists $h>0$ such that for each compact set $K\subseteq {\mathbb R}$ we have $T\ast \varphi \in {\mathrm A},$ $\varphi \in \mathcal{D}^{M_p,h}_K.$ If  $(M_p)$ additionally satisfies (M.3), then there exist $l>0$ and two elements 
$f,\ g\in {\mathrm A}$ such that $T=P(D)f+g.$
\end{itemize}

Now we will focus our attention to the case that ${\mathrm A}=AAP({\mathbb R} : X),$ resp. ${\mathrm A}=AAA({\mathbb R} : X).$ Then ${\mathrm A}$ is closed under the uniform convergence and addition, and we have ${\mathrm A}\subseteq \mathcal{D}^{\prime *}_{{\mathrm A}}(X)$ (\cite{nova-mono}). Hence, as a special case of the above assertions, we have the following lemma:

\begin{lem}\label{mare-basitinjo}
Let $(M_{p})$ 
satisfy the conditions \emph{(M.1), (M.2)} and \emph{(M.3'),} and let $T\in 
{\mathcal D}^{ \prime (M_{p})}(X),$ resp., $T\in
{\mathcal D}^{ \prime \{M_{p}\}}(X).$ Then the following holds:
\begin{itemize}
\item[(i)]
Suppose that there exist an ultradifferential operator $P(D)=\sum_{p=0}^{\infty}a_p D^p$
of class $(M_p)$, resp., of
class $\{M_p\},$ and $f,\ g\in \mathcal{D}^{\prime *}_{AAP({\mathbb R} : X)}(X),$ resp. $f,\ g\in \mathcal{D}^{\prime *}_{AAA({\mathbb R} : X)}(X),$  such that $T=P(D)f+g.$ Then
$T\in \mathcal{D}^{\prime *}_{AAP({\mathbb R} : X)}(X),$ resp. $T\in \mathcal{D}^{\prime *}_{AAA({\mathbb R} : X)}(X).$
\item[(ii)] If $T\in 
{\mathcal D}_{L^{1}}^{ \prime }((M_{p}) : X)$ and $T\ast \varphi \in AAP({\mathbb R} : X),$ $\varphi \in {\mathcal D}^{(M_{p})},$ resp. $T\ast \varphi \in AAA({\mathbb R} : X),$ $\varphi \in {\mathcal D}^{(M_{p})},$ then there is a number $h>0$ such that for each compact set $K\subseteq {\mathbb R}$ we have $T\ast \varphi \in AAP({\mathbb R} : X),$ $\varphi \in \mathcal{D}^{M_p,h}_K,$ resp. $T\ast \varphi \in AAA({\mathbb R} : X),$ $\varphi \in \mathcal{D}^{M_p,h}_K.$ 
\item[(iii)] Assume that $T\in \mathcal{D}^{\prime (M_{p})}(X)$ and there is a number $h>0$ such that for each compact set $K\subseteq {\mathbb R}$ we have $T\ast \varphi \in AAP({\mathbb R} : X),$ $\varphi \in \mathcal{D}^{M_p,h}_K,$ resp. $T\ast \varphi \in AAA({\mathbb R} : X),$ $\varphi \in \mathcal{D}^{M_p,h}_K.$ If  $(M_p)$ additionally satisfies \emph{(M.3),} then there exist $l>0$ and bounded functions
$f,\ g\in AAP({\mathbb R} : X),$ resp. $f,\ g\in AAA({\mathbb R} : X),$ such that $T=P_{l}(D)f+g.$
\end{itemize}
\end{lem}

We need one more lemma, which is probably known in the existing literature.

\begin{lem}\label{beograd-leka} (Supremum formula)
Let $f\in AA({\mathbb R} : X).$ Then for each $a\in {\mathbb R}$ we have $\|f\|_{\infty}=\sup_{x\geq a}\|f(x)\|.$
\end{lem}

\begin{proof}
Essentially, we need only to prove that for each $a,\ y \in {\mathbb R}$ and $\epsilon>0$ we have
\begin{align}\label{strudla}
\|f(y)\| \leq  \sup_{x\geq a}\|f(x)\|+\epsilon.
\end{align}
By definition of almost automorphy, for the sequence $(b_{n}:=n)$ we can extract a subsequence $(a_{n})$ of it such that $f(y)=\lim_{n\rightarrow \infty}\lim_{k\rightarrow \infty}f(y-a_{k}+a_{n}).$ This, in particular, implies that we can find two integers $n_{0}(\epsilon),\, k_{0}(\epsilon)  \in {\mathbb N}$ such that, for every $k\geq k_{0}(\epsilon),$
we have $\| f(y-a_{k}+a_{n})-f(y)\|\leq \epsilon.$ This gives \eqref{strudla} and finishes the proof of lemma. 
\end{proof}

Now we are able to state the following important result:

\begin{thm}\label{slep-blind-auto-ultra}
Let $(M_{p})$ 
satisfy the conditions \emph{(M.1), (M.2)} and \emph{(M.3)',}
and let $T\in {\mathcal D}^{\prime \ast}_{L^{1}}(M_{p} : X).$ Consider the following assertions:
\begin{itemize}
\item[(i)] There exist an element $S\in {\mathcal D}^{\prime \ast}_{L^{1}}(M_{p} :X)$, a number $l>0$ in the Beurling case/a sequence $(r_{p})\in {\mathrm R}$ in the Roumieu case, and bounded functions $f,\ g\in AAP({\mathbb R} : X),$ resp.  $f,\ g\in AAA({\mathbb R} : X),$ such that $S=P_{l}(D)f+g,$ resp.  $S=P_{r_{p}}(D)f+g,$ and $S=T$ on $[0,\infty).$
\item[(ii)] There exist a number $l>0,$ resp. a sequence $(r_{p})\in {\mathrm R},$ and bounded functions $f,\ g\in AAP({\mathbb R} : X),$ resp.   $f,\ g\in AAA({\mathbb R} : X),$ such that $T=P_{l}(D)f+g,$ resp. $T=P_{r_{p}}(D)f+g,$ on $[0,\infty).$
\item[(iii)] There exist an ultradifferential operator $P(D)=\sum_{p=0}^{\infty}a_p D^p$
of $\ast$-class and bounded functions $f_{1},\ f_{2}\in AAP({\mathbb R} : X),$ resp. $f_{1},\ f_{2}\in AAA({\mathbb R} : X),$ such that $T=P(D)f_{1}+f_{2}$  on $[0,\infty).$
\item[(iv)] There exist an element $S\in {\mathcal D}^{\prime \ast}_{L^{1}}(M_{p} :X)$, an ultradifferential operator $P(D)=\sum_{p=0}^{\infty}a_p D^p$
of $\ast$-class and bounded functions $f_{1},\ f_{2}\in AAP({\mathbb R} : X),$ resp. $f_{1},\ f_{2}\in AAA({\mathbb R} : X),$ such that $S=P(D)f_{1}+f_{2}$  and $S=T$ on $[0,\infty).$
\item[(v)] $T\in B^{\prime \ast}_{AAP}(X),$ resp. $T\in B^{\prime \ast}_{AAA}(X),$ there exist an element $S\in {\mathcal D}^{\prime \ast}_{L^{1}}(M_{p} :X)$, an ultradifferential operator $P(D)=\sum_{p=0}^{\infty}a_p D^p$
of $\ast$-class and bounded functions $f_{1},\ f_{2}\in AAP({\mathbb R} : X),$ resp. $f_{1},\ f_{2}\in AAA({\mathbb R} : X),$ such that $S=P(D)f_{1}+f_{2}$  and $S=T$ on $[0,\infty).$
\item[(vi)] $T\in B^{\prime \ast}_{AAP}(X),$ resp. $T\in B^{\prime \ast}_{AAA}(X),$ and there exists an element $S\in {\mathcal D}^{\prime \ast}_{L^{1}}(M_{p} :X)$ such that  $S=T$ on $[0,\infty)$ and $S\ast \varphi \in AAP({\mathbb R} : X),$ $\varphi \in {\mathcal D}^{\ast},$ resp. $S\ast \varphi \in AAA({\mathbb R} : X),$ $\varphi \in {\mathcal D}^{\ast}.$
\item[(vii)] $T\in B^{\prime \ast}_{AAP}(X),$ resp. $T\in B^{\prime \ast}_{AAA}(X).$
\item[(viii)] $T \ast \varphi \in AAP({\mathbb R} : X),$ $\varphi \in {\mathcal D}^{\ast}_{0},$ resp. $T \ast \varphi \in AAA({\mathbb R} : X),$ $\varphi \in {\mathcal D}^{\ast}_{0}.$
\item[(ix)] There exists a sequence $(T_{n})$ in ${\mathcal E}^{\ast}(X) \cap AAP({\mathbb R} : X),$ resp. ${\mathcal E}^{\ast}(X) \cap AAA({\mathbb R} : X),$ such that $\lim_{n\rightarrow \infty}T_{n}=T$ in ${\mathcal D}_{L^{1}}^{\prime \ast}(X).$
\item[(x)] $T \ast \varphi \in AAP({\mathbb R} : X),$ $\varphi \in {\mathcal D}^{\ast},$ resp. $T \ast \varphi \in AAA({\mathbb R} : X),$ $\varphi \in {\mathcal D}^{\ast}.$
\end{itemize}
Then we have \emph{(i)} $\Rightarrow $ \emph{(ii)} $\Rightarrow $ \emph{(iii)} $\Rightarrow $ \emph{(iv)} $\Rightarrow $ \emph{(v)} $\Rightarrow $ \emph{(vi)} $\Leftrightarrow $ \emph{(vii)} $\Leftrightarrow $ \emph{(viii)} $\Leftrightarrow $ \emph{(ix)} $\Leftrightarrow $ \emph{(x)}. Furthermore, if $(M_{p})$  
additionally satisfies the condition \emph{(M.3),} then the assertions \emph{(i)}-\emph{(x)} are equivalent for the Beurling class.
\end{thm}

\begin{proof}
For the sake of brevity, we will consider only the asymptotically almost periodic case for the Beurling class (in the case of almost automorphy, it is only worth noting 
that, for the proof of implication (ix) $\Rightarrow $ (vii), we need to use the supremum formula deduced in Lemma \ref{beograd-leka}).  
The implications (i) $\Rightarrow $ (ii) $\Rightarrow $ (iii) $\Rightarrow $ (iv) and (vi) $\Rightarrow$ (vii) are trivial, while the implication (vii)
$\Rightarrow$ (viii) can be deduced as it has been done for scalar-valued distributions (cf. the proof of implication (i) $\Rightarrow$ (ii) of \cite[Theorem I]{ioana-asym}).
In order to see that (iv) implies (v), we only need to prove that the assumptions of (iv) imply $T\in B^{\prime (M_{p})}_{AAP}(X).$
Suppose that $f_{i}(\cdot)=g_{i}(\cdot)+q_{i}(\cdot)$ on $[0,\infty),$ where $f_{i}\in AP({\mathbb R} : X)$ and $g_{i}\in C_{0}([0,\infty) : X)$ ($i=1,2$).
Then $T_{1}:=P(D)g_{1}+g_{2}$ is an almost periodic ultradistribution of Beurling class due to \cite[Theorem 2]{mat-bilten} (see \cite[Theorem 3.1]{msd-nsjom} for almost automorphic case), and all that we need to show is $T_{2}:=P(D)(f_{1}-g_{1})+(f_{2}-g_{2})\in B^{\prime (M_{p})}_{AAP}(X),$ i.e., 
\begin{align*}
\lim_{h\rightarrow +\infty} \Bigl \langle P(D)\bigl(f_{1}-g_{1}\bigr)+\bigl(f_{2}-g_{2}\bigr)  , \varphi (\cdot -h)  \Bigr \rangle=0,\quad \varphi \in {\mathcal D}^{(M_{p})} .
\end{align*}
Towards this end, assume that $-\infty <a<b<+\infty$ and supp$(\varphi) \subseteq [a,b].$ Let $\epsilon>0$ be given. Then there 
exist a sufficiently large finite number $h_{0}(\epsilon)>0$ and a sufficiently large finite number $c_{\varphi}>0,$ independent of $\epsilon ,$ such that, for every $h\geq h_{0}(\epsilon),$ we have the following (cf. also the proof of \cite[Theorem 1]{mat-bilten}):
\begin{align*}
& \Biggl\| \Bigl \langle P(D) \bigl(f_{1}-g_{1}\bigr)+\bigl(f_{2}-g_{2}\bigr)  , \varphi (\cdot -h)  \Bigr \rangle \Biggr\|
\\ & =\Biggl\| \sum_{p=0}^{\infty}(-1)^{p}a_{p}\int^{\infty}_{-\infty}\bigl( f_{1}(t)-g_{1}(t)\bigr)\varphi^{(p)}(t-h)\, dt +\int^{\infty}_{-\infty}\bigl( f_{2}(t)-g_{2}(t)\bigr)\varphi (t-h)\, dt \Biggr\|
\\& =\Biggl\| \sum_{p=0}^{\infty}(-1)^{p}a_{p}\int^{b}_{a}\bigl( f_{1}(t+h)-g_{1}(t+h)\bigr)\varphi^{(p)}(t)\, dt +\int^{b}_{a}\bigl( f_{2}(t)-g_{2}(t)\bigr)\varphi (t)\, dt \Biggr\|
\\ &= \Biggl\| \sum_{p=0}^{\infty}(-1)^{p}a_{p}\int^{b}_{a}q_{1}(t+h)\varphi^{(p)}(t)\, dt +\int^{b}_{a}q_{2}(t)\varphi (t)\, dt \Biggr\|
\\ & \leq \epsilon \Biggl[\sum_{p=0}^{\infty}\bigl|a_{p} \bigr| \bigl\|\varphi^{(p)}\bigr\|_{L^{1}}+ \bigl\|\varphi\bigr\|_{L^{1}}\Biggr]
 \leq \epsilon c_{\varphi}.
\end{align*}
This yields (v). The implication (v) $\Rightarrow$ (vi) follows by applying Lemma \ref{mare-basitinjo}. In order to see that (vii) implies (vi), assume that $T=T_{ap}+Q$ on $[0,\infty),$ where $T_{ap}$ and $Q$ satisfy the requirements of Definition \ref{rk-mbo}. Put $S:=T_{ap}+Q$ on ${\mathbb R}.$ Then we need to prove that $S\ast \varphi \in AAP({\mathbb R} : X),$ $\varphi \in {\mathcal D}^{\ast},$ i.e., that
$\lim_{h\rightarrow +\infty}(Q\ast \varphi)(h)=0,$ $\varphi \in {\mathcal D}^{\ast}.$ But, this follows from the fact that $(Q\ast \varphi)(h)=\langle Q, \check{\varphi}(\cdot-h)\rangle=\langle Q_{h}, \check{\varphi} \rangle \rightarrow 0$ as $h\rightarrow +\infty$ dy definition of space $B^{\prime (M_{p})}_{+,0}(X).$
We can show that (viii) implies (ix) as in the proof of implication of (ii) $\Rightarrow$ (vii) in Theorem \ref{slep-blind}. The implication (ix) $\Rightarrow$ (x) follows similarly as in the first part of the proof of \cite[Proposition 7]{buzar-tress}, while the implication (x) $\Rightarrow$ (vii) is trivial.
Now we will prove the implication (ix) $\Rightarrow $ (vii). Let $(T_{n})$ be a sequence in ${\mathcal E}^{(M_{p})}(X) \cap AAP({\mathbb R} : X)$ such that $\lim_{n\rightarrow \infty}T_{n}=T$ in ${\mathcal D}_{L^{1}}^{\prime (M_{p})}(X),$ and let $T_{n}=G_{n}+Q_{n}$ on $[0,\infty)$ for some $G_{n}\in AP({\mathbb R} : X)$ 
and $Q_{n}\in C_{b}({\mathbb R} : X)$ tending to zero at plus infinity ($n\in {\mathbb N}$). Let
$B$ be a bounded subset of ${\mathcal D}_{L^{1}}^{\prime (M_{p})}(X),$ and let $\epsilon>0$ be given. It can be simply shown that the set $B':=B\cup \{\varphi( \pm \cdot \pm h) : h\in {\mathbb R},\ \varphi \in B\}$ is a bounded subset of ${\mathcal D}_{L^{1}}^{\prime (M_{p})}(X),$ as well. Then $(G_{n}-G_{m})+(Q_{n}-Q_{m})=T_{n}-T_{m},$ $m,\ n\in {\mathbb N},$ so that there exists an integer $n_{0}(\epsilon)\in {\mathbb N}$ such that, for every $m,\ n\in {\mathbb N}$ with $\min(m,n)\geq n_{0}(\epsilon),$ we have:
$$
\sup_{\varphi \in B'}\Bigl \| \bigl \langle G_{n}-G_{m},\varphi \bigr \rangle +  \bigl \langle Q_{n}-Q_{m},\varphi \bigr \rangle \Bigr \| \leq \epsilon ;
$$
especially,
\begin{align}\label{utakmica-pre}
\sup_{\varphi \in B',\, h\in {\mathbb R}}\Bigl \| \bigl \langle G_{n}-G_{m},\varphi(h-\cdot) \bigr \rangle +  \bigl \langle Q_{n}-Q_{m},\varphi (h-\cdot) \bigr \rangle \Bigr \| \leq \epsilon .
\end{align}
By the foregoing, we have $\lim_{h\rightarrow +\infty}\|\langle Q_{n}-Q_{m},\varphi (h-\cdot) \rangle \| =0,$ for  every $m,\ n\in {\mathbb N}$ with $\min(m,n)\geq n_{0}(\epsilon),$ so that there exists a sufficiently large number $M_{\epsilon}>0$ such that
$$
\Bigl \| \bigl \langle G_{n}-G_{m},\varphi (h-\cdot) \bigr \rangle \Bigr \| \leq \epsilon,\quad m,\, n\geq  n_{0}(\epsilon),\  h\geq M_{\epsilon},\ \varphi \in B'.
$$
By  Lemma \ref{beograd-leka}  and the almost periodicity of function on the left hand side of the above inequality, we get
\begin{align}\label{utakmica}
\Bigl \| \bigl \langle G_{n}-G_{m},\varphi (h-\cdot) \bigr \rangle \Bigr \| \leq \epsilon,\quad m,\, n\geq  n_{0}(\epsilon),\  h\in {\mathbb R},\ \varphi \in B'.
\end{align}
This clearly implies that the sequence $(\langle G_{n}, \varphi \rangle )$ is Cauchy in $X$ and therefore convergent ($\varphi \in {\mathcal D}_{L^{1}}$). 
Set $\langle G, \varphi \rangle :=\lim_{m\rightarrow \infty}\langle G_{m}, \varphi \rangle$ for all $\varphi \in {\mathcal D}_{L^{1}}$.
Letting $m\rightarrow \infty$ in \eqref{utakmica}, we obtain
\begin{align}\label{utakmicalj}
\Bigl \| \bigl \langle G_{n}-G,\varphi (h-\cdot) \bigr \rangle \Bigr \| \leq \epsilon,\quad  n\geq  n_{0}(\epsilon),\  h\in {\mathbb R},\ \varphi \in B'.
\end{align}
This, in turn, implies that $G$ is the limit of convergent sequence $(G_{n})$ in ${\mathcal D}_{L^{1}}^{\prime (M_{p})}(X)$ and therefore $G\in {\mathcal D}_{L^{1}}^{\prime (M_{p})}(X).$
Furthermore, \eqref{utakmicalj} shows that $\lim_{n\rightarrow \infty}G_{n}\ast \varphi =G\ast \varphi ,$ $\varphi \in {\mathcal D}^{(M_{p})}$ in $C_{b}({\mathbb R} :  X),$ so that
$G\ast \varphi \in AP({\mathbb R} : X),$ $\varphi \in {\mathcal D}^{(M_{p})};$ hence, $G\in B^{\prime \ast}_{AP}(X).$ 
Define $\langle Q, \varphi \rangle :=\lim_{m\rightarrow \infty}\langle T-G_{m}, \varphi \rangle$ for all $\varphi \in {\mathcal D}_{L^{1}}$. Then $Q\in {\mathcal D}_{L^{1}}^{\prime (M_{p})}(X)$ and all that we need to show is that $Q$ tends to zero at plus infinity. For this, observe first that $Q$ has to be the limit of a Cauchy sequence $(Q_{n})$ in ${\mathcal D}_{L^{1}}^{\prime (M_{p})}(X)$ and that combining \eqref{utakmica-pre} and \eqref{utakmica} yields
\begin{align}\label{utakmica-hajka}
\Bigl \| \bigl \langle Q_{n}-Q_{m},\varphi (h-\cdot) \bigr \rangle \Bigr \| \leq 2\epsilon,\quad m,\, n\geq  n_{0}(\epsilon),\  h\in {\mathbb R},\ \varphi \in B'.
\end{align}
Due to \eqref{utakmica-hajka}, we obtain
\begin{align*}
\Bigl \| \bigl \langle Q_{n}-Q,\varphi (h-\cdot) \bigr \rangle \Bigr \| \leq 2\epsilon,\quad n\geq  n_{0}(\epsilon),\  h\in {\mathbb R},\ \varphi \in B'.
\end{align*}
With $n= n_{0}(\epsilon),$ the above simply yields the existence of a number $h_{0}(\epsilon)>0$ such that 
\begin{align*}
\Bigl \| \bigl \langle Q,\varphi (h-\cdot) \bigr \rangle \Bigr \| \leq 3\epsilon,\quad h\geq h_{0},\ \varphi \in B'.
\end{align*}
Hence, (vii) is proved.
Finally, if $(M_{p})$ 
satisfies the condition (M.3), then the implication (x) $\Rightarrow$ (i) follows by applying Lemma \ref{mare-basitinjo}(ii)-(iii) with $S=T$, and therefore, the assertions (i)-(x) are mutually equivalent for the Beurling class. 
\end{proof}

Using the equivalence of parts (vii) and (x) in Theorem \ref{slep-blind-auto-ultra}, as well as the continuity properties of convolution, it readily follows that $B^{\prime \ast}_{AAP}(X),$ resp. $B^{\prime \ast}_{AAA}(X),$ is a closed subspace of ${\mathcal D}^{\prime \ast}_{L^{1}}(M_{p} : X).$ Concerning Theorem \ref{slep-blind-auto-ultra}, we have the following observations:

\begin{rem}\label{basit-kil-osama}
The implication (vii) $\Rightarrow$ (x) of Theorem \ref{slep-blind-auto-ultra} can be proved directly, as explained below. Let $T=T_{ap}+Q$ on $[0,\infty),$ with $T_{ap}$ and $Q$ being the same as in Definition \ref{rk-mbo}. Since 
$$
(T\ast \varphi)(h)=\bigl(T_{ap}\ast \varphi\bigr)(h)+\Bigl( \bigl[ T-T_{ap}\bigr] \ast \varphi \Bigr)(h),\quad h\in {\mathbb R},\ \varphi \in {\mathcal D}^{\ast},
$$
and $T_{ap}\ast \varphi \in AP({\mathbb R} :X),$ $\varphi \in {\mathcal D}^{\ast},$ we need to prove that
$$
\lim_{h\rightarrow +\infty}\Bigl( \bigl[ T-T_{ap}\bigr] \ast \varphi \Bigr)(h)=0,\quad \varphi \in {\mathcal D}^{\ast}
$$
i.e., that:
\begin{align}\label{osamosam}
\lim_{h\rightarrow +\infty}\bigl \langle T-T_{ap} , \varphi(h-\cdot) \bigr \rangle=0,\quad \varphi \in {\mathcal D}^{\ast}.
\end{align}
Since supp$(\varphi(h-\cdot))\subseteq [0,\infty)$ for $h\geq \sup(\mbox{supp}(\varphi)),$ we have that \eqref{osamosam} is equivalent with
\begin{align*}
\lim_{h\rightarrow +\infty}\bigl \langle Q, \varphi(h-\cdot) \bigr \rangle=0,\quad \varphi \in {\mathcal D}^{\ast}.
\end{align*}
But, this follows from the equality $\langle Q, \varphi(h-\cdot)  \rangle =\langle Q_{h}, \check{\varphi}\rangle$ for all $\varphi \in {\mathcal D}^{\ast},$ and the fact that $Q\in B^{\prime \ast}_{+,0}(X).$
\end{rem}

\begin{rem}\label{basit-kil-osam-gomez}
Assume that $(M_{p})$ additionally satisfies (M.3). Then,
for the Beurling class, the assumptions in (ii) are equivalent to say that $F$ and $X$ are asymptotically almost periodic, resp. asymptotically almost automorphic (see Theorem \ref{slep-blind-auto-ultra}). 
We feel duty bound to say that, in this case, it seems very plausible that the assertion of Theorem \ref{slep-blind-auto-ultra} can be slightly generalized for $\omega$-ultradistributions (see \cite{gomezinjo} for more details). The interested reader may also try to 
prove analogues of \cite[Theorem 4.3.1, p. 123; Theorem 4.3.2, p. 124; Theorem 4.5.1, p. 128]{cheban} for asymptotically almost periodic (automorphic) ultradistributions of $\ast$-class.
\end{rem}

The assertion of \cite[Proposition 2]{ioana-asym} can be formulated for vector-valued almost periodic distributions and vector-valued almost periodic ultradistributions of Beurling class, with the corresponding sequence $(M_{p})$ satisfying (M.1), (M.2) and (M.3).
Now we would like to state the following automorphic version of \cite[Proposition 2]{ioana-asym} (vector-valued case):

\begin{prop}\label{low}
\begin{itemize}
\item[(i)] Consider the equation $X^{\prime}=f \cdot X +U$ in ${\mathcal D}^{\prime}(X),$ where $f\in {\mathcal E} \cap AA({\mathbb R} : {\mathbb C})$ and $U\in B^{\prime}_{AA}(X).$ If $T\in B^{\prime}_{AAA}(X)$ is a solution of this equation on $[0,\infty),$ then its almost automorphic part is  a solution of this equation on ${\mathbb R}.$   
\item[(ii)] Let $(M_{p})$ satisfy \emph{(M.1)}, \emph{(M.2)} and \emph{(M.3).} Consider the equation $X^{\prime}=f \cdot X +U$ in ${\mathcal D}^{\prime (M_{p})}(X),$ where $f\in {\mathcal E}^{(M_{p})} \cap AA({\mathbb R} : {\mathbb C})$ and $U\in B^{\prime (M_{p})}_{AA}(X).$ If $T\in B^{\prime (M_{p})}_{AAA}(X)$ is a solution of this equation on $[0,\infty),$ then its almost automorphic part is  a solution of this equation on ${\mathbb R}.$   
\end{itemize}
\end{prop}

\begin{proof}
Since the pointwise product of an almost automorphic $X$-valued function with a scalar-valued almost automorphic function is again an 
almost automorphic $X$-valued function, we can use the structural characterization of $B^{\prime}_{AA}(X),$ resp. $B^{\prime (M_{p})}_{AA}(X),$  in order to see that these spaces are closed under pointwise multiplication with a scalar-valued function $f\in {\mathcal E} \cap AA({\mathbb R} : {\mathbb C}),$ resp. $f\in {\mathcal E}^{(M_{p})} \cap AA({\mathbb R} : {\mathbb C});$ furthermore, the almost automorphy of a bounded distribution, resp. bounded ultradistribution of Beurling class, is equivalent to say that, for every real sequence $(b_{n}),$ there exist a subsequence $(a_{n})$ of $(b_{n})$ and a vector-valued distribution $S \in {\mathcal D}^{\prime}(X)$, resp. ultradistribution $S \in {\mathcal D}^{\prime (M_{p})}(X),$ such that $\lim_{n\rightarrow \infty}\langle T_{a_{n}},\varphi \rangle=\langle S,\varphi \rangle$ for all $\varphi \in {\mathcal D},$ resp. $\varphi \in {\mathcal D}^{(M_{p})},$ and $\lim_{n\rightarrow \infty}\langle S_{-a_{n}},\varphi \rangle=\langle T, \varphi \rangle$ for all $\varphi \in {\mathcal D}$, resp. $\varphi \in {\mathcal D}^{(M_{p})};$ see \cite{msd-nsjom}. Using these facts, we can repeat almost literally the proof of \cite[Proposition 2]{ioana-asym} to deduce the validity of (i) and (ii); the only thing worth noting is that we do not need Bochner's criterion for the proof because the limits appearing on lines 10 and -4 on p. 258 of \cite{ioana-asym} follows, actally, from the almost automorphy of corresponding distributions. 
\end{proof}

\section[An application to systems of ordinary differential equations...]{An application to systems of ordinary differential equations in distribution and ultradistribution spaces}

Let $n\in {\mathbb N},$ and let $A=[a_{ij}]_{1\leq i,\ j\leq n}$ be a given complex matrix such that $\sigma(A) \subseteq \{z\in {\mathbb C}: \Re z < 0 \}.$ In this section, we analyze the existence of asymptotically almost periodic (automorphic) solutions of equation
\begin{align}\label{ackno}
X^{\prime}=AX+F,\quad X\in {\mathcal D}^{\prime}(X^{n})\ \mbox{ on } \ [0,\infty)
\end{align} 
and equation
\begin{align}\label{ackno-ultra}
X^{\prime}=AX+F,\quad X\in {\mathcal D}^{\prime \ast}(X^{n}) \ \mbox{ on } \ [0,\infty),
\end{align} 
where $F$ is an asymptotically almost periodic (automorphic) $X^{n}$-valued distribution in \eqref{ackno} and $F$ is an asymptotically almost periodic (automorphic) $X^{n}$-valued ultradistribution of $\ast$-class in \eqref{ackno-ultra}. By a solution of \eqref{ackno}, resp. \eqref{ackno-ultra}, we mean any distribution $X\in {\mathcal D}^{\prime}(X^{n}),$ resp. ${\mathcal D}^{\prime \ast}(X^{n}),$
such that \eqref{ackno}, resp. \eqref{ackno-ultra}, holds in distributional, resp. ultradistributional, sense on $[0,\infty).$

The main result of this section reads as follows:

\begin{thm}\label{egzist-ultra}
\begin{itemize}
\item[(i)] Let $F=[F_{1} \ F_{2} \cdot \cdot \cdot F_{n}]^{T}\in B^{\prime}_{AAP}(X^{n}),$ resp. 
$F=[F_{1} \ F_{2} \cdot \cdot \cdot F_{n}]^{T}\in B^{\prime}_{AAA}(X^{n}).$
Then there exists a solution
$X=[X_{1} \ X_{2} \cdot \cdot \cdot X_{n}]^{T}\in B^{\prime}_{AAP}(X^{n}),$ resp. $X=[X_{1} \ X_{2} \cdot \cdot \cdot X_{n}]^{T}\in B^{\prime}_{AAA}(X^{n})$ of \eqref{ackno}.
Furthermore, any distributional solution $X$ of \eqref{ackno} belongs to the space $B^{\prime}_{AAP}(X^{n}),$ resp. 
$B^{\prime}_{AAA}(X^{n}).$
\item[(ii)] Let $(M_{p})$ 
satisfy the conditions \emph{(M.1), (M.2)} and \emph{(M.3)',}
and let $F=[F_{1} \ F_{2} \cdot \cdot \cdot F_{n}]^{T}\in {\mathcal D}^{\prime \ast}_{L^{1}}(M_{p} : X^{n})$ be such that, for every $i\in {\mathbb N}_{n},$ there exist an ultradifferential operator $P_{i}(D)=\sum_{p=0}^{\infty}a_{i,p}D^{p}$
of $\ast$-class and bounded functions $f_{1,i},\ f_{2,i}\in AAP({\mathbb R} : X),$ resp. $f_{1,i},\ f_{2,i}\in AAA({\mathbb R} : X),$ for $1\leq i\leq n,$ such that $F_{i}=P_{i}(D)f_{1,i}+f_{2,i}$  on $[0,\infty).$ 
Then there exist ultradifferential operators $P_{ij}(D)$
of $\ast$-class, bounded functions $h_{ij}\in AAP({\mathbb R} : X),$ resp. $h_{ij}\in AAA({\mathbb R} : X)$ ($1\leq i,\ j\leq n$),
and bounded functions
$h_{2,i}\in AAP({\mathbb R} : X),$ resp. $h_{2,i}\in AAA({\mathbb R} : X)$ ($1\leq i\leq n$), such that $X=[X_{1} \ X_{2} \cdot \cdot \cdot X_{n}]^{T}\in B^{\prime \ast}_{AAP}(X^{n}),$ resp. $X=[X_{1} \ X_{2} \cdot \cdot \cdot X_{n}]^{T}\in B^{\prime \ast}_{AAA}(X^{n}),$ is a solution of \eqref{ackno-ultra}, where
$X_{i}=\sum_{j=1}^{n}P_{ij}(D)h_{ij} +h_{2,i}$ for $i\in {\mathbb N}_{n}.$ Furthermore, any ultradistibutional solution $X$ of $\ast$-class to the equation \eqref{ackno-ultra} has such a form.
\end{itemize}
\end{thm}

\begin{proof}
We will prove only (ii). Since $\sigma(A) \subseteq \{z\in {\mathbb C}: \Re z < 0 \}$, \cite[Lemma 1]{buzar-tressinjo} and a careful inspection of the proof of \cite[Theorem 3, pp. 117-118]{buzar-tressinjo} yield that it is sufficient to prove the required assertion in one-dimensional case. So, let 
$\Re \lambda<0$ and let us consider the equation
\begin{align}\label{ackno-ultra-prx}
X^{\prime}=\lambda X+P(D)f+g,\quad X\in {\mathcal D}^{\prime \ast}(X)\ \mbox{ on } \ [0,\infty),
\end{align} 
with given bounded functions $f,\ g\in AAP({\mathbb R} : X),$ resp. $f,\ g\in AAA({\mathbb R} : X),$ and $P(D)=\sum_{p=0}^{\infty}a_{p}D^{p}.$ Then $(e^{-\lambda \cdot}X)^{\prime}=e^{-\lambda \cdot}[P(D)f+g]$ on $[0,\infty),$ which implies $\langle e^{-\lambda \cdot}X,\varphi \rangle =-\langle e^{-\lambda \cdot}[P(D)f+g],I(\varphi) \rangle +\langle \mbox{Const.}, \varphi \rangle ,$ $\varphi \in {\mathcal D}^{\ast}_{0}.$ Therefore, there exists a constant $c\in {\mathbb C}$ such that
\begin{align*}
\langle X,\varphi \rangle &= \Biggl \langle P(D)f+g , e^{-\lambda \cdot}\int^{\cdot}_{-\infty}\Bigl[ e^{\lambda t}\varphi(t)-\nu(t)\int^{\infty}_{-\infty}e^{\lambda u}\varphi (u)\, du\Bigr] \, dt \Biggr \rangle 
\\ &+c\int^{\infty}_{-\infty}e^{\lambda u}\varphi(u)\, du,\quad \varphi \in {\mathcal D}^{\ast}_{0}.
\end{align*}
Applying product rule for the term containing the function $f(\cdot),$ as well as the Fubini theorem and partial integration for the term containing the function $g(\cdot),$ we get:
\begin{align*}
\langle X,\varphi \rangle &=\sum_{p=0}^{\infty}(-1)^{p+1}a_{p}\int^{\infty}_{-\infty}f(t)\Bigl[ e^{-\lambda \cdot}I(e^{\lambda \cdot}\varphi)(\cdot)\Bigr]^{(p)}(t)\, dt +c\int^{\infty}_{-\infty}e^{\lambda u}\varphi(u)\, du
\\ &+\int^{\infty}_{-\infty}g(t)e^{-\lambda t} \int^{t}_{-\infty}\Bigl[e^{\lambda u}\varphi(u)-\nu(u)\int^{\infty}_{-\infty}e^{\lambda v}\varphi (v)\, dv\Bigr] \, du \, dt
\\ & =\sum_{p=0}^{\infty}a_{p}\int^{\infty}_{-\infty}f(t)\sum_{k=1}^{p}(-1)^{k+1}{p \choose k}\lambda^{p-k}e^{-\lambda t}\Biggl[\sum_{j=0}^{k-1}\binom{k-1}{j}\lambda^{k-j}e^{\lambda t}\varphi^{(j)}(t)
\\ &-\nu^{(k-1)}(t)\int^{\infty}_{-\infty}e^{\lambda u}\varphi(u)\, du \Biggr]
\\ & -\sum_{p=0}^{\infty}a_{p}\int^{\infty}_{-\infty}f(t) \lambda^{p}e^{-\lambda t}\int^{t}_{-\infty}\Biggl[ e^{\lambda u}\varphi(u) -\nu (t) \int^{\infty}_{-\infty}e^{\lambda v}\varphi (v)\, dv\Biggr]\, du \, dt
\\ &
+c\int^{\infty}_{0}e^{\lambda u}\varphi(u)\, du
+\int^{\infty}_{0}\Biggl[\int^{t}_{0}e^{\lambda (t-s)}g(s)\, ds\Biggr] \varphi(t)\, dt
\\& - \int^{\infty}_{0}e^{\lambda v}\Biggl[ \int^{\infty}_{-\infty}\nu(t)\int^{t}_{0}e^{-\lambda s}g(s)\, ds \, dt\Biggr] \varphi(v)\, dv
, \quad \varphi \in {\mathcal D}^{\ast}_{0}.
\end{align*}
Applying again the Fubini theorem for the term containing the part $\nu^{(k-1)}(\cdot)$ as well as the Fubini theorem and partial integration for the term containing the part $f(t) \lambda^{p}e^{-\lambda t},$ we get 
\begin{align*}
\langle X,\varphi \rangle &=\int^{\infty}_{0}f(t) \Biggl[ \sum_{p'=p+1}^{\infty}a_{p'}\sum_{k=p}^{p'}(-1)^{k+1}{p' \choose k}{k-1 \choose p}\lambda^{p'-p} \Biggr] \varphi^{(p)}(t)\, dt
\\&-\int^{\infty}_{0}e^{\lambda v} \Biggl[ \sum_{p=0}^{\infty}a_{p}\int^{2}_{1}f(t)\sum_{k=1}^{p} (-1)^{k+1}{p\choose k}\lambda^{p-k}e^{-\lambda t}\nu^{(k-1)}(t)\, dt\Biggr] \varphi(v)\, dv
\\&-\int^{\infty}_{0}
\Biggl[\sum_{p=0}^{\infty}a_{p}\lambda^{p}\int^{t}_{0}e^{\lambda (t-s)}f(s)\, ds\Biggr] \varphi(t)\, dt
\\& +\int^{\infty}_{0}e^{\lambda v}\Biggl[\sum_{p=0}^{\infty}a_{p}\lambda^{p}\int^{\infty}_{-\infty}\nu(t)\int^{t}_{0}e^{-\lambda s}f(s)\, ds \, dt \Biggr]\varphi(v)\, dv
+c\int^{\infty}_{0}e^{\lambda u}\varphi(u)\, du
\\& +\int^{\infty}_{0}\Biggl[\int^{t}_{0}e^{\lambda (t-s)}g(s)\, ds\Biggr] \varphi(t)\, dt
\\& - \int^{\infty}_{0}e^{\lambda v}\Biggl[ \int^{\infty}_{-\infty}\nu(t)\int^{t}_{0}e^{-\lambda s}g(s)\, ds \, dt\Biggr] \varphi(v)\, dv
, \quad \varphi \in {\mathcal D}^{\ast}_{0}.
\end{align*}
The series in the first and fourth addend converge, whereas the functions $\int^{\cdot}_{0}e^{\lambda (\cdot-s)}g(s)\, ds,$ $ce^{\lambda \cdot}$ are asymptotically almost periodic, resp. asymptotically almost automorphic, due to our assumption $\Re \lambda <0$ (see \cite{nova-mono} for more details).
The series $\sum_{p=0}^{\infty}a_{p}\lambda^{p}\int^{\infty}_{-\infty}\nu(t)\int^{t}_{0}e^{-\lambda s}f(s)\, ds \, dt$ converges uniformly for $t\geq 0$ and, to complete the whole proof, it suffices to show that the series 
$$
\sum_{p=0}^{\infty}a_{p}\int^{2}_{1}f(t)\sum_{k=1}^{p} (-1)^{k+1}{p\choose k}\lambda^{p-k}e^{-\lambda t}\nu^{(k-1)}(t)\, dt
$$
is absolutely convergent. Towards this end, observe that (M.3') yields the existence of a finite constant $c'>0$ such that
the absolute value of 
$$
a_{p}\int^{2}_{1}f(t)\sum_{k=1}^{p} (-1)^{k+1}{p\choose k}\lambda^{p-k}e^{-\lambda t}\nu^{(k-1)}(t)\, dt
$$
does not exceed $\|f\|_{\infty}\mbox{Const.}|a_p|\sum_{k=0}^{p}\frac{M_{k}}{h^{k}}\leq \|f\|_{\infty}\mbox{Const.}|a_p|(p+1)m^{p}\frac{M_{p}}{h^{p}},$ where $m:=\sum_{p=1}^{\infty}1/m_{p}<\infty.$ This completes the proof in a routine manner.
\end{proof}

\section*{Acknowledgement}
The author is partially supported by
grant 174024 of Ministry of Science and Technological Development,
Republic of Serbia.

\end{document}